\documentclass[letterpaper, 10 pt, conference]{ieeeconf}  % Comment this line out if you need a4paper

\IEEEoverridecommandlockouts                              % This command is only needed if 
\usepackage[utf8]{inputenc} % allow utf-8 input
\usepackage[T1]{fontenc}    % use 8-bit T1 fonts
\usepackage{hyperref}       % hyperlinks
\usepackage{url}            % simple URL typesetting
\usepackage{booktabs}       % professional-quality tables
\usepackage{amsfonts}       % blackboard math symbols
\usepackage{nicefrac}       % compact symbols for 1/2, etc.
\usepackage{microtype}      % microtypography
\usepackage{xcolor}         % colors

\usepackage{amsmath,amssymb,amsthm}
\usepackage{mathtools}
\usepackage{float}

\DeclareMathOperator*{\argmax}{arg\,max}

\newcommand{\R}{\mathbb{R}}

\newcommand{\Jcal}{\mathcal{J}}
\newcommand{\T}{^{\top}}
\newcommand{\invT}{^{-\top}}
\newcommand{\norm}[1]{\left\lVert#1\right\rVert}
\newcommand{\inner}[2]{\left\langle #1, #2 \right\rangle}
\newcommand{\sigmin}{\sigma_{\min}}
\newtheorem{theorem}{Theorem}

\newtheorem{assumption}{Assumption}

\newtheorem{lemma}{Lemma}

\title{\LARGE \bf
End-to-End Training of High-Dimensional Optimal Control with Implicit Hamiltonians via Jacobian-Free Backpropagation}

\author{%
    Eric Gelphman*
    \and
    Deepanshu Verma*\thanks{* denotes co-first author. Corresponding author: \url{eric_gelphman@mines.edu}}
    \and
    Nicole Tianjiao Yang 
    \and
    Stanley Osher 
    \and 
    Samy Wu Fung
}

\begin{document}

\maketitle
\begin{abstract}
    Neural network approaches that parameterize value functions have succeeded in approximating high-dimensional optimal feedback controllers when the Hamiltonian admits explicit formulas. 
    However, many practical problems, such as the space shuttle reentry problem and bicycle dynamics, among others, may involve implicit Hamiltonians that do not admit explicit formulas, limiting the applicability of existing methods. 
    Rather than directly parameterizing controls, which does not leverage the Hamiltonian's underlying structure, we propose an end-to-end implicit deep learning approach that directly parameterizes the value function to learn optimal control laws. Our method enforces physical principles by ensuring trained networks adhere to the control laws by exploiting the fundamental relationship between the optimal control and the value function's gradient; this is a direct consequence of the connection between Pontryagin's Maximum Principle and dynamic programming. Using Jacobian-Free Backpropagation (JFB), we achieve efficient training despite temporal coupling in trajectory optimization. We show that JFB produces descent directions for the optimal control objective and experimentally demonstrate that our approach effectively learns high-dimensional feedback controllers across multiple scenarios involving implicit Hamiltonians, which existing methods cannot address.
\end{abstract}

\thispagestyle{empty}
\pagestyle{empty}

%%%%%%%%%%%%%%%%%%%%%%%%%%%%%%%%%%%%%%%%%%%%%%%%%%%%%%%%%%%%%%%%%%%%%%%%%%%%%%%%

% \title{Jacobian-Free Backpropagation for High-Dimensional Optimal Control with Implicit Hamiltonians}

% \title{Efficient Training of High-Dimensional Feedback Controllers for Non-Analytic Hamiltonians via Jacobian-Free Backpropagation}

% The \author macro works with any number of authors. There are two commands
% used to separate the names and addresses of multiple authors: \And and \AND.
%
% Using \And between authors leaves it to LaTeX to determine where to break the
% lines. Using \AND forces a line break at that point. So, if LaTeX puts 3 of 4
% authors names on the first line, and the last on the second line, try using
% \AND instead of \And before the third author name.

    \section{Introduction}\label{sec: intro}
    We consider the problem of generating feedback controllers for high-dimensional optimal control problems of the form 
    \begin{equation}
        \begin{split}
        \min_{u \in U} \int_0^T L(s, z_{x}, u) ds + G(z_{x}(T)), \\ \text{subject to} \quad \dot{z}_{x} = f(t, z_{x},u), \;\; z_{x}(0) = x,
        \end{split}
        \label{eq:originalOC}
    \end{equation}
    where $z_x\in \R^n$   is the state trajectory with dynamics $f$ and initial condition $x$, $u(t)\in U\subset \R^m$ is the control, $L$ is the running cost, and $G$ is the terminal cost. The subscript in  $z_{x}$ denotes the dependence of the state on the initial condition $x$. 
    The Pontryagin Maximum Principle (PMP)~\cite{pontryagin2018mathematical, kopp1962pontryagin} provides necessary conditions for optimality through the generalized Hamiltonian
    \begin{equation}
        \mathcal{H}(t,z_{x},p_{x}, u) = - \langle p_{x}, f(t,z_{x},u) \rangle - L(t,z_{x},u),
    \end{equation}
    where $p_{x}$ is the adjoint variable. 
    At the optimal controller $u^\star$, we have 
    % system must be satisfied:
    % % \begin{subequations}      
    \begin{align}
        \dot{z}_{x} &= -\nabla_p \mathcal{H}(t,z_{x}, p_{x} , u^\star), \quad z_{x}(0) = x \\
        \dot{p}_{x} &= \nabla_z \mathcal{H}(t,z_{x}, p_{x}, u^\star), \quad p_{x}(T) = \nabla G(z_{x}(T)), \label{eq:adjoint_dyn}
    \end{align}
    along with the optimality condition
    \begin{equation}
        u^\star \in  \argmax_u \mathcal{H}(t,z_{x},p_{x}, u) \implies \nabla_u \mathcal{H}(t,z_{x}, p_{x}, u) = 0.
        \label{eq:u_star_def_general}
    \end{equation}
    % \label{eq:pmp_dynamics}
    % \end{subequations}
    
    A powerful strategy for solving this system involves connecting the PMP to the Hamilton-Jacobi-Bellman (HJB) equation, where the adjoint is linked to the gradient of the value function. Neural networks that parameterize the value function have demonstrated success in solving high-dimensional optimal control problems~\cite{onken2022neural, onken2021neural, lin2021alternating}, particularly when the Hamiltonian
    \begin{equation}
        H(t,z_{x},p_{x}) = \sup_u \mathcal{H}(t,z_{x},u,p_{x}),
        \label{eq:straight_H}
    \end{equation}
    and the corresponding optimal controllers $u^\star$ in~\eqref{eq:u_star_def_general} admit closed-form solutions~\cite{onken2022neural, onken2021neural, nakamura2021adaptive, zhao2024offline,Verma-HJB-RL}. This value function parameterization approach is particularly attractive because it leverages the PMP to extract additional structural information (see section~\ref{sec:OC_back}), making the training process significantly more efficient compared to  directly parameterizing the controller~\cite{onken2022neural, onken2021neural, li2024neural}.
    However, the requirement of a closed-form solution for the maximization in \eqref{eq:u_star_def_general} forms a critical bottleneck. Many problems of practical interest lack this convenient closed-form property. The Hamiltonian often fails to admit closed-form solutions in applications such as the space shuttle reentry problem~\cite[Example 4.1]{betts2010practical}, 
    % passenger aircraft landing scenarios~\cite{bulirsch1991abort, bulirsch1991abort2}, 
    and bicycle model control~\cite{xiao2023safe}, among numerous others~\cite{betts2010practical, dickmanns1972maximum, ge2007neural, ge1999adaptive, ge2003neural, ge2002adaptive, wang2006iss, calise2001adaptive}. This limitation is not restricted to exotic dynamics; even when the dynamics are affine in the control input, a non-quadratic running cost $L$ in $u$ can result in a Hamiltonian with no closed-form solution.
    
    \textbf{Our Contribution:} We propose an efficient end-to-end approach for training neural networks in high-dimensional optimal control problems where the value function is parameterized but explicit formulas for the Hamiltonian in~\eqref{eq:straight_H} are unavailable. 
    % While one could simply parameterize the control input $u$, this approach fails to leverage the rich problem structure available in the Hamiltonian. 
    Our key insight is to embed the parameterized value function within an implicit network architecture~\cite{el2021implicit, bai2019deep, heaton2023explainable} that defines the optimal control through the optimality condition itself; this circumvents the need for closed-form solutions (Section~\ref{subsec: problem_formulation}).  To address the computational challenges of differentiating through a time integral of fixed points, we leverage Jacobian-Free Backpropagation (JFB)~\cite{fung2022jfb, bolte2024one, geng2021training} to reduce training complexity while maintaining theoretical guarantees. We extend JFB theory from~\cite{fung2022jfb} to the optimal control setting and show that our approach produces descent directions despite the temporal coupling inherent in trajectory optimization. This framework preserves the advantages of value function parameterization while expanding applicability to problems where Hamiltonians do not have closed-form solutions. All codes are provided in \url{https://github.com/mines-opt-ml/jfb-for-implicit-oc}.

    \section{Related Works}\label{sec:related_works}
    Recent advances in neural networks have enabled efficient solutions to high-dimensional optimal control problems when the Hamiltonian admits closed-form solutions. 
    Approaches such as Neural-PMP~\cite{gu2022pontryagin}, PMP-Net, and Pontryagin Differentiable Programming~\cite{jin2020pontryagin, jin2021safe} leverage the Pontryagin Maximum Principle to create end-to-end differentiable frameworks for learning optimal controllers. Value function parameterization methods~\cite{ruthotto2020machine, li2024neural, onken2022neural, onken2021ot, onken2021neural, Verma-HJB-RL, lin2021alternating} have shown particularly strong results by directly parameterizing the value function with neural networks and recovering the optimal controller through PMP relations. However, all these methods fundamentally rely on the ability to analytically solve the Hamiltonian maximization problem, limiting their applicability when closed-form solutions are unavailable. Building on differentiable optimization foundations, several approaches have addressed computational challenges in learning-based control. DiffMPC~\cite{amos2018differentiable} differentiates through Model Predictive Control using KKT conditions, while IDOC~\cite{xu2023revisiting} achieves linear time complexity through direct matrix equation evaluation. Learned MPC approaches use neural approximations for computational efficiency~\cite{hertneck2018learning}, though they typically abandon optimal control structure. Our work extends value function parameterization methods to implicit Hamiltonians by leveraging implicit neural networks and Jacobian-Free Backpropagation.

    \section{Background}
    % We begin with some background results that we will employ in our approach.
    \subsection{Optimal Control} 
    \label{sec:OC_back}
    
    To derive a feedback controller, a powerful approach is to utilize the system's value function, $\phi(t,z)$. A fundamental result in optimal control theory establishes that the value function $\phi$ contains complete information about the optimal control. Specifically, the gradient of the value function yields the adjoint variable, see \cite[Theorem I.6.2]{fleming2006controlled}:
    \begin{equation}\label{eq:gradPhi}
    p_{x}(t)=\nabla_{z} \phi \big(t,z_{x}^\star(t) \big).
    \end{equation}
    
    Consequently, we can express the optimal control from \eqref{eq:u_star_def_general} directly as a feedback law in terms of the value function:
    \begin{equation}\label{eq:opt_control_nec}
    \begin{split}
      u^\star(t)=u^\star\left(t,z^\star_{x}(t),\nabla_{z} \phi \big(t,z_{x}^\star(t) \big) \right).
    \end{split}
    \end{equation}

    Herein lies the crucial bottleneck this work addresses. Evaluating the feedback law requires solving the maximization problem \eqref{eq:u_star_def_general}. When this maximization does not admit a closed from solution, which is often the case in practice~\cite{betts2010practical}, computing $u^\star$ becomes non-trivial and and existing approaches~\cite{onken2022neural, onken2021neural, meng2025recent,ruthotto2020machine, lin2021alternating} face computational intractability when high-dimensional controllers are present.
    
    \subsection{Implicit Deep Learning}
    \label{subsec: implicit_deep_learning}
    In our optimal control framework, we use an Implicit Neural Network (INN) architecture to tackle the implicit Hamiltonian challenge. Unlike traditional networks, INN outputs are defined by an implicit (or fixed point) condition~\cite{el2021implicit, fung2024generalization}. For our problem, this operator, $T_\theta$, is derived directly from the Hamiltonian's optimality condition~\eqref{eq:u_star_def_general}, and its fixed point, $u_\theta^\star$, represents the optimal control. That is, the output of a network is the fixed point 
    \begin{equation}
        u_\theta^\star = T_{\theta}(u_\theta^\star; t,z),
        \label{eq:fixed_point}
    \end{equation}
    where $\theta \in \mathbb{R}^p$ refers to the parameters of the network and $(t,z)$ are the input features.
    
     INNs have been applied to domains as diverse as image classification \cite{bai2020multiscale}, inverse problems \cite{gilton2021deep, Yin2022Learning, liu2022online, heaton2021feasibility}, game theory \cite{mckenzie2024three}, maze-solving~\cite{knutson2024logical}, and decision-focused learning~\cite{mckenzie2024differentiating}. 
    Since the output of $u^\star$ in~\eqref{eq:u_star_def_general} is implicitly defined, INNs are naturally suited for modeling $u^\star$, as they are not defined via an explicit sequence of layers, but rather by an implicit, fixed-point, condition. 
    This condition can be viewed as specifying when the problem is considered solved. 
    
    To train INNs, one generally differentiates through the solution of a fixed point operator. 
    A common approach to train implicit networks differentiates through the fixed point equation implicitly~\cite{he2016deep, el2021implicit}. That is, we can differentiate both sides of~\eqref{eq:fixed_point} to obtain 
    \begin{equation*}
        \frac{du_\theta^\star}{d\theta}(t,z) = \frac{\partial T_\theta(u_\theta^\star; t,z)}{\partial u} \frac{d u_\theta^\star}{d\theta}(t,z) + \frac{\partial T_\theta(u_\theta^\star; t,z)}{\partial \theta}.        
    \end{equation*}
    One can then isolate the derivative of interest to obtain 
    \begin{equation}
        \frac{du_\theta^\star}{d\theta}(t,z) = \Bigg( \underbrace{I - \frac{\partial T_\theta(u_\theta^\star; t,z)}{\partial u}}_{= \mathcal{J}_\theta}\Bigg)^{-1} \frac{\partial T_\theta(u_\theta^\star; t,z)}{\partial \theta}.
        % \quad \text{ where } \quad \mathcal{J}_\theta = \left(I - \frac{\partial T_\theta(u_\theta^\star; t,z)}{\partial u}\right).
        \label{eq:implicit_gradient}
    \end{equation}
    A primary challenge when training INNs for feedback controllers is the computational burden of solving a linear system for each evaluation of $(t,z)$, that is, \emph{for each sample path, for each time step, and for every epoch}. This difficulty is exacerbated in our setting, where these computations must also occur \emph{at each time step} during trajectory generation; this is because the controller must be evaluated at each time step during the trajectory generation.
    
    Consequently, recent work has focused on improving training efficiency through various approaches~\cite{bai2019deep, el2021implicit, fung2022jfb, bolte2024one}. Among these, Jacobian-Free Backpropagation (JFB)~\cite{fung2022jfb} stands out perhaps the simplest; this approach shares many similarities with one-step differentiation~\cite{bolte2024one}. \emph{JFB replaces the $\mathcal{J}_\theta$ with an identity matrix, implementing a zeroth-order Neumann expansion of the inverse term}. That is, 
    \begin{equation}
        \frac{du_\theta^\star}{d\theta}(t,z) \approx \frac{\partial T_\theta(u_\theta^\star; t,z)}{\partial \theta}.
        \label{eq: jfb_gradient}
    \end{equation}
    
    This straightforward approach has been proven effective and is simple to implement~\cite[Figure 3]{fung2022jfb}. 
    Additional analysis; however, is required to make JFB provably work in our setting (see Section~\ref{subsec: JFB_analysis}).
    
    \section{High-Dimensional Feedback Controllers with Implicit Hamiltonians}
    We consider optimal control problems where the Hamiltonian is implicit, making direct application of classical feedback control methods challenging. To address this, we follow recent advances in high-dimensional control solvers~\cite{onken2022neural, onken2021neural} and leverage the key relationship from~\eqref{eq:gradPhi} to parameterize the value function directly.
    
    \subsection{Problem Formulation and Implicit Network Parameterization}
    \label{subsec: problem_formulation}
    We formulate the training problem as an expectation over initial conditions:
    \begin{subequations}
    \begin{align}
        % \begin{split}
        \min_\theta \; \mathbb{E}_{x \sim \rho} \; J_x(\theta) &= \int_0^T L(s, z_{x}, u^\star_{\theta}) ds + G(z_{x}(T)), \label{eq:training_problem1}
    \end{align}
        \begin{align}
        \text{subject to:} \quad \dot{z}_x &= f(t,z_x,u_\theta^\star), \quad z_x(0) = x, \label{eq:training_problem2}
        \\
        u_\theta^\star &\in \argmax_u \mathcal{H}(t,z,\nabla \phi_\theta, u), \label{eq:training_problem3}
        % \end{split}
        % \label{eq:full_opt}
    \end{align}
    \end{subequations}
    where $\rho \in \mathbb{P}$ is a distribution of initial conditions.
    % Following~\cite{onken2022neural, onken2021neural}, we consider the entire time horizon (setting $t=0$) and omit time subscripts for clarity.
    Note here that once~\eqref{eq:training_problem1}-\eqref{eq:training_problem3} is solved, we have access to the semi-global value function (and hence, a \emph{feedback controller}) since we train on a distribution of initial conditions. 
    
    The central challenge lies in computing (and differentiating through) $u_\theta^\star$ efficiently, since the Hamiltonian maximization problem generally lacks a closed-form solution. Implicit networks provide an elegant solution by leveraging the equivalence between fixed-point and optimization problems: we can define our control network implicitly through the optimality condition itself, which circumvents the need for explicit solutions. 

    Since $u_\theta^\star$ must satisfy the first-order optimality condition $\nabla_u \mathcal{H}(t,z,\nabla \phi_\theta, u_\theta^\star) = 0$, we construct a fixed point operator that naturally converges to this condition via gradient ascent:
    \begin{equation}
    \label{eq:T_theta}
        T_\theta(u; t, z_{x}) = u + \alpha \nabla_u \mathcal{H}(t,z_x,\nabla \phi_\theta, u),
    \end{equation}
    where the fixed point operator $T_\theta$ depends on network parameters $\theta$ through the parameterized value function $\phi_\theta$. This choice ensures that fixed points of $T_\theta$ satisfy the necessary optimality conditions while naturally respecting the underlying optimal control structure. Other fixed point operators satisfying the optimality conditions are possible~\cite{ryu2022large, mckenzie2024three, mckenzie2024differentiating}, but gradient ascent provides the simplest formulation without affecting our core JFB methodology.
    
    A key insight of this approach is leveraging~\eqref{eq:gradPhi} to use the gradient of the value function as the dual variable, eliminating the need to learn the costate/adjoint. This strategy has proven particularly effective for learning feedback controllers via value function methods~\cite{li2024neural, onken2022neural, onken2021ot, vidal2023taming}.

    \subsection{Efficient Training via Jacobian-Free Backpropagation}

    Training the implicit network presents significant computational challenges due to the nested optimization structure inherent in our formulation. At each time step, for every trajectory, and across all training epochs, we must both evaluate and differentiate through the fixed point $u^\star_\theta$. This requirement stems from computing gradients of the objective function $J_x(\theta)$ in~\eqref{eq:training_problem1}-\eqref{eq:training_problem3}, which necessitates backpropagating through entire trajectories where each point depends on solving the implicit optimization problem in~\eqref{eq:training_problem3}.

    The computational burden scales as $\mathcal{O}(N_{\text{batch}} \cdot N_t \cdot m^3)$ \emph{per epoch} after discretization, where $N_{\text{batch}}$ is the batch size, $N_t$ is the number of time steps, and $m$ is the control dimension. The prohibitive $m^3$ term arises from solving the linear system in~\eqref{eq:implicit_gradient} at \emph{each evaluation point} $(t,z)$, making traditional implicit differentiation computationally intractable for realistic problem sizes. 
    To address this bottleneck, we employ Jacobian-Free Backpropagation (JFB) as described in~\eqref{eq: jfb_gradient}, which \emph{eliminates the matrix inversion in~\eqref{eq:implicit_gradient} and reduces the complexity to $\mathcal{O}(N_{\text{batch}} \cdot N_t \cdot m^2)$}, which is crucial when $m$ is large and evaluations occur at every time step of every trajectory as is the case here. 
    This substantial reduction from cubic to quadratic scaling in the control dimension makes training feasible for high-dimensional control problems. 
    We now provide theoretical foundations that justify the use of JFB for solving~\eqref{eq:training_problem1}-\eqref{eq:training_problem3}.

    \subsubsection{Analysis of JFB for Optimal Control with Implicit Hamiltonians}
    \label{subsec: JFB_analysis}

    While JFB provides substantial computational savings, establishing that it produces descent directions for our training objective requires careful analysis due to the unique structure of optimal control problems. Unlike the original JFB framework~\cite{fung2022jfb}, which considers a single fixed-point problem, our formulation involves a continuum of fixed points varying continuously over time through the integral in~\eqref{eq:training_problem1}. Each time point requires solving a distinct instance of~\eqref{eq:implicit_gradient}, and the overall gradient depends on the accumulated effect across the entire trajectory. Consequently, we cannot simply recycle existing JFB results to establish descent properties, and a more comprehensive analysis that accounts for the temporal coupling inherent in optimal control problems is required. We now demonstrate that despite this added complexity, JFB retains its descent properties under some additional smoothness and boundedness assumptions.

    We begin with the two core assumptions from the original work on JFB~\cite{fung2022jfb} that are used to show descent when JFB is applied to a single fixed-point subproblem.
    \begin{assumption}
    There exists $\gamma \in  (0,1)$ such that $T_\theta$ in \eqref{eq:T_theta} is $\gamma$-contractive in $u$ for all $t \in [0,T], z \in \mathbb{R}^n$ and $ \theta \in \mathbb{R}^p$. The operator $T_{\theta}$ is continuously differentiable with respect to $\theta, t, u, z$.
    \label{assumption_1}
    \end{assumption}

    The above contraction assumption of $T_\theta$ can be satisfied for a sufficiently small step size $\alpha > 0$ and a well-behaved $\mathcal{H}$. A standard proof can be found in \cite[Chapter 5]{beck2017first}. 

    \begin{assumption}
For any $\theta, t, u, z$, the matrix $M_\theta = \frac{\partial T_{\theta}}{\partial \theta}(u; t, z)$ satisfies the singular value bounds
\begin{equation*}
\sqrt{\frac{\gamma}{\beta}} \leq \sigma_{\min}(M_\theta) \leq \sigma_{\max}(M_\theta) \leq \frac{1}{\sqrt{\beta}}
\end{equation*}
where $\beta > 0$ and $\gamma > 0$ is the contraction factor of $T_\theta$.
\label{assumption_2}
\end{assumption}
    Note that $M_\theta(\cdot)$ depends on $\theta, t, u, z$ but for brevity, we follow the notation convention of \cite{fung2022jfb}. This assumption ensures that $M_\theta$ has full row rank, is well-conditioned, and allows us to show that $(M_\theta M_\theta^\top)^{-1}$ has uniform bounds on its eigenvalues using the relationship between singular values and eigenvalues of symmetric positive definite matrices.
    \begin{lemma}
    \label{lemma: lambdaA}
    $\left( M_\theta M_\theta^\top \right)^{-1}$ has uniform upper and lower bounds on the eigenvalues for all $t,z,u,\theta$. That is, $\exists$  positive constants $0< \lambda_- < \lambda_+$, such that $\lambda_{-} \: I \preceq \left( M_\theta M_\theta^\top \right)^{-1} \preceq \lambda_+ \: I$, for all $t,z,u,\theta$.
    \end{lemma}
    \begin{proof}
        From Assumption~\ref{assumption_2}, we have $\sigma_{\max}(M_\theta) \leq \frac{1}{\sqrt{\beta}}$, which gives
        $$\lambda_{\max}(M_\theta M_\theta^\top) = \sigma_{\max}^2(M_\theta) \leq \frac{1}{\beta}.$$
        Therefore, $\lambda_{\min}((M_\theta M_\theta^\top)^{-1}) = \frac{1}{\lambda_{\max}(M_\theta M_\theta^\top)} \geq \beta$ for all $(\theta, t, u, z)$.
        
        Similarly, from $\sigma_{\min}(M_\theta) \geq \sqrt{\frac{\gamma}{\beta}}$, we have
        $$\lambda_{\min}(M_\theta M_\theta^\top) = \sigma_{\min}^2(M_\theta) \geq \frac{\gamma}{\beta}.$$
        Therefore, $\lambda_{\max}((M_\theta M_\theta^\top)^{-1}) = \frac{1}{\lambda_{\min}(M_\theta M_\theta^\top)} \leq \frac{\beta}{\gamma}$ for all $(\theta, t, u, z)$.
        
        Hence, for
        \begin{align*}
        \lambda_- &:= \inf_{(\theta, t, u, z)} \lambda_{\min}((M_\theta M_\theta^\top)^{-1}) \geq \beta \\
        \lambda_+ &:= \sup_{(\theta, t, u, z)} \lambda_{\max}((M_\theta M_\theta^\top)^{-1}) \leq \frac{\beta}{\gamma},
        \end{align*}
        we have the result.
        \end{proof}
    % \begin{proof}
    % % $\lambda_-$ has been specified in Assumption~\ref{assumption_3} (2), we demonstrate $\lambda_+$ in this proof.
    %  % $\lambda_{\min}$ and $\lambda_{+}$ denote the smallest and largest eigenvalue of a matrix, respectively.
    % For any $t,z,u,\theta$, it holds that $
    % \lambda_{\min}\left((M_\theta M_\theta^\top)^{-1} \right) = {1/ \lambda_{\max}\left(M_\theta M_\theta^\top\right)}\geq \beta
    % $ by Assumption~\ref{assumption_3}.
    % By Assumption~\ref{assumption_3}, $\lambda_{\min}(M_\theta M_\theta^\top) > \gamma \lambda_{\max}(M_\theta M_\theta^\top) = \frac{\gamma}{\beta}$. Then, $\lambda_{\max}((M_\theta M_\theta^\top)^{-1}) = \frac{1}{\lambda_{\min}((M_\theta M_\theta^\top))} < \frac{\beta}{\gamma}$.
    % Hence, $\lambda_- \coloneqq \inf_{(\theta, t, u, z)} \lambda_{\min}((M_\theta M_\theta^\top)^{-1}) \geq \beta$, $\lambda_+ \coloneqq \sup_{(\theta, t, u, z)} \lambda_{\max}((M_\theta M_\theta^\top)^{-1}) \leq \frac{\beta}{\gamma}$.
    % \end{proof}

    Assumptions~\ref{assumption_1}-\ref{assumption_2} alone are insufficient for our setting since JFB is applied continuously throughout trajectory generation when evaluating~\eqref{eq:training_problem1}-\eqref{eq:training_problem3}. To address this, we first establish the mathematical relationship between the true gradient and its JFB approximation.

    The true derivative of $J_x(\theta)$ and its JFB approximation are given by 
    \begin{align}
        \frac{dJ_x(\theta)}{d\theta} = \int_0^T v_\theta(t) dt, \quad \text{ and } \quad 
        \frac{d\tilde{J}_x(\theta)}{d\theta} = \int_0^T w_\theta(t) dt,
        \label{eq:full_derivatives}
    \end{align}
    respectively, where 
    \begin{equation}\label{eq:integrand_definitions}
        \begin{split}
        &v_\theta(t) = \frac{d u_\theta^\star}{d \theta}^{\top}(\nabla_u L(t,z_{x}, u_\theta^\star) + \nabla_u f^{\top}p_{x}), \\
        &w_\theta(t) = \frac{\partial T_\theta}{\partial  \theta}^{\top}(\nabla_u L(t,z_{x}, u_\theta^\star) + \nabla_u f^{\top}p_{x}).
        \end{split}
    \end{equation} 
    Here, $\dfrac{du_\theta^\star}{d\theta}$ is the implicit gradient given by~\eqref{eq:implicit_gradient}, and $p_x$ satisfies a corresponding adjoint equation (see~\cite[Eqn.~(6.11)]{evans1983introduction}). The key computational advantage of the JFB gradient is that \emph{it circumvents the expensive computation of $\frac{du_\theta^\star}{d\theta}$}, resulting in significantly reduced computational cost.
    To establish descent properties in the optimal control setting, we require some additional assumptions that ensure the derivatives of the control problem remain well-behaved.

    \begin{assumption}
    There exists $\eta > 0$ such that for all $t,z,\theta$, $\|\nabla_uL(t,z_x,u_\theta^\star) + \nabla_uf^{\top}p_x\| \geq \eta$. 
    \label{assumption_dhdu}
    \end{assumption}
    Assumption~\ref{assumption_dhdu} requires the gradient of the Hamiltonian with respect to the control be bounded away from zero. This prevents degenerate cases where the control has no meaningful effect on the objective and ensures that parameter updates actually improve performance.

        \newcommand{\lemmaPositiveInnerProduct}[1]{
    % Let $\lambda_{+}$ and $\lambda_{-}$ be the uniform bound of the largest and smallest eigenvalues of $M_\theta M_\theta^\top$ over $t,z,u$, respectively.
    Under Assumptions \ref{assumption_1}-\ref{assumption_dhdu}, $ \langle v_\theta(t),w_\theta(t)\rangle \geq \|M_\theta v_\theta\|^2 (\lambda_- - \gamma \lambda_+) , \forall t,z,u,\theta$. Here, $v_\theta, w_\theta$ are defined in \eqref{eq:integrand_definitions}, $\lambda_-, \lambda_+$ are the uniform bounds of $\left( M_\theta M_\theta^\top \right)^{-1}$ as specified in Lemma~\ref{lemma: lambdaA}.}
    \begin{lemma}
    \label{lemma:positive_inner_product}
    \lemmaPositiveInnerProduct{main}
    \end{lemma}
    
    Lemma~\ref{lemma:positive_inner_product} along with the following assumption allows us to show the main theorem in this work. All proofs are provided in the appendix.
    
    \begin{assumption} Let 
    \begin{equation*}
        C_v = \frac{1}{T} \int_0^T v_\theta(t) dt \quad \text{ and } \quad C_w = \frac{1}{T} \int_0^T w_\theta(t) dt, 
    \end{equation*}
    where $v_\theta, w_\theta$ are defined in~\eqref{eq:integrand_definitions}. We have that for all $z,u,\theta$,
    \begin{enumerate}
        \item $v_\theta, w_\theta$ are entrywise $L^2[0,T]$ with respect to $t$,
        \item $\|v_\theta(t) - C_v\|< \|M_{\theta}v_{\theta}\|\sqrt{\lambda_{-} - \gamma\lambda_{+}} \quad $ and $ \quad \|w_\theta - C_w\| < \|M_{\theta}w_{\theta}\|\sqrt{\lambda_{-} - \gamma\lambda_{+}}$, where $M_\theta$ is defined in Assumption~\ref{assumption_2}, 
$\lambda_-, \lambda_+$ are the uniform bounds of $\left( M_\theta M_\theta^\top \right)^{-1}$ in Lemma~\ref{lemma: lambdaA},
        and $\gamma$ is the contraction factor of $T_\theta$.
    \end{enumerate}
    \label{assumption_timeaverage}
    \end{assumption}
Let $c\coloneqq \frac{ \gamma \eta}{\beta(1+\gamma)}$, where $\beta$, $\gamma$, $\eta$ are the constants from the previous assumptions. A sufficient condition for the second part of Assumption~\ref{assumption_timeaverage} to hold is $\|v_\theta(t) - C_v\|, \|w_\theta - C_w\| < c \sqrt{\lambda_{-} - \gamma\lambda_{+}}$, as detailed in the proof of Lemma~\ref{lemma:positive_inner_product} in the appendix.
    Assumption~\ref{assumption_timeaverage} states that the gradient components must be well-behaved over time, neither blowing up nor oscillating too wildly. 
    Together, these assumptions ensure that the integrand of $\frac{dJ_x}{d\theta}$ remains bounded and that the matrix $M_\theta M_\theta^\top \mathcal{J}_\theta^{-\top}$ maintains sufficient conditioning to prevent numerical instabilities. Under these conditions, we can establish the primary theoretical result of this work. 
    \newcommand{\mainTheorem}[1]{
    Under Assumptions \ref{assumption_1}-\ref{assumption_timeaverage}, the (negative) JFB approximation of the gradient given by 
    \begin{equation*}
        -\frac{d\tilde{J}_x(\theta)}{d\theta} = -\int_0^T w_\theta(t) dt
    \end{equation*} 
    is a descent direction for $J_x(\theta)$ with respect to $\theta$, where $w_\theta$ is defined in~\eqref{eq:integrand_definitions}.
    }
    \begin{theorem}
    \label{theorem:main_theorem}
    \mainTheorem{main}
    \end{theorem}

    \section{Experiments}
    
    We evaluate our proposed method on three applications of optimal control problems that give rise to implicit Hamiltonians and test the possible scalability.

    \noindent \textbf{Quadrotor with Exponential Running Cost}. Our first test case involves quadcopter dynamics with affine control dependence in the system dynamics $f$. We choose an exponential running cost $L = \exp(\|u\|^2)$ paired with a quadratic terminal cost $G(z(T)) = \|z(T) - z_{\text{target}}\|^2$ for some final state $z_{\text{target}}$. This exponential cost structure creates a non-linear relationship between control and the Hamiltonian, leading to an implicit Hamiltonian without analytical solutions. 

    \begin{figure}[t]
        \centering
        \textbf{Quadrotor} 
        \\
        \includegraphics[width=1.0\linewidth]{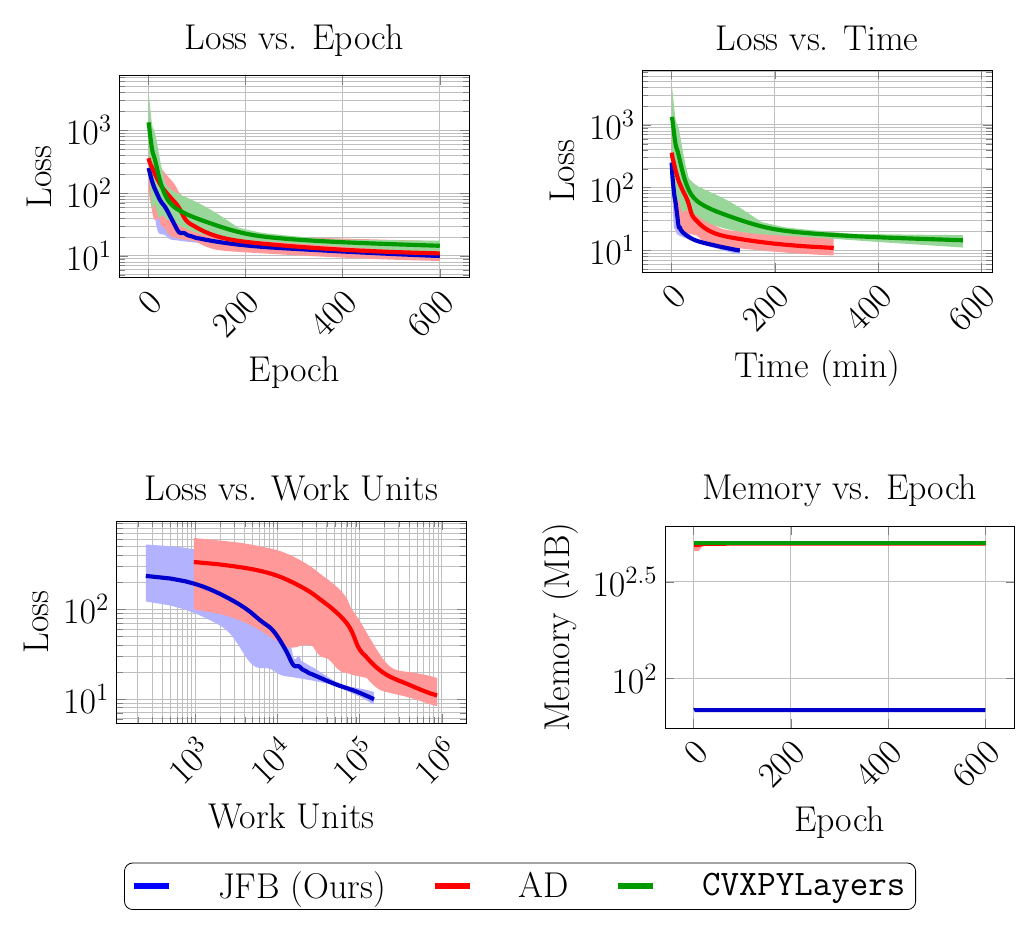}
        \caption{\small{Comparison of JFB, automatic differentiation (AD), and \texttt{CVXPYLayers}~\cite{agrawal2019differentiable} (Implicit Differentiation) for training the value function (and hence, feedback controller) for a quadrotor across three metrics. (Top Left) Loss versus training epochs. (Top Right) Loss plotted against cumulative runtime in minutes. (Bottom Left) Loss plotted against cumulative work units, with one work unit being one evaluation of $\frac{\partial T_{\theta}}{\partial \theta}$, which is equivalent to backpropagation through one application of $T_\theta$. (Bottom Right) Maximum GPU memory usage per training epoch.}}
        \label{fig:single_quadrotor}
    \end{figure}  

    \begin{figure}[t]
        \centering
        \textbf{5 Bicycles} 
        \\
        \includegraphics[width=1.0\linewidth]{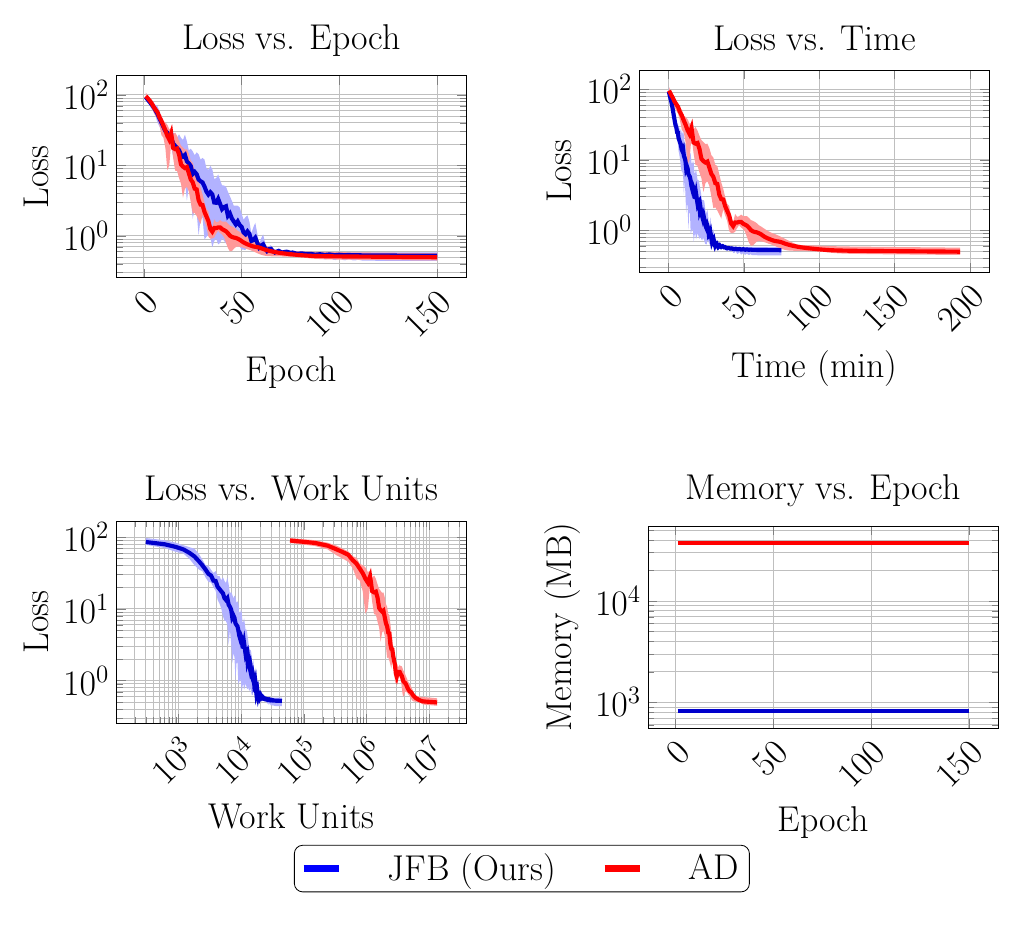}
        \caption{\small{Comparison of JFB and automatic differentiation (AD) for training a feedback for 5 bicycles across four metrics. (Top Left) Loss versus training epochs. (Top Right) Loss plotted against cumulative runtime in minutes. (Bottom Left) Loss plotted against cumulative work units. (Bottom Right) Maximum GPU memory usage per training epoch.}}
        \label{fig:5_bikes}
    \end{figure}  

    \begin{figure}[t]
        \centering
        \textbf{20 Bicycles} 
        \\
        \includegraphics[width=1\linewidth]{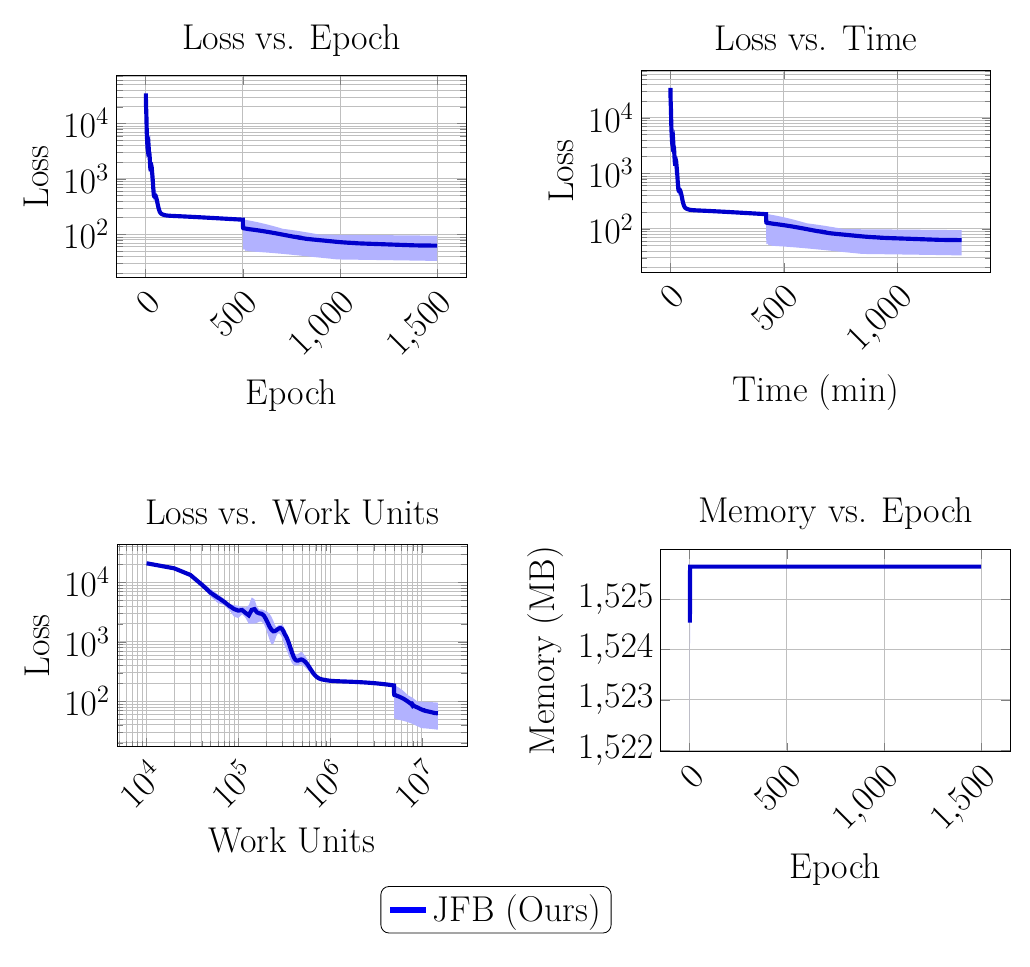}
        \caption{\small{Results for high-dimensional 20-bicycle problem using JFB. (Top Left) Loss vs. training epochs. (Top Right) Loss vs. runtime in minutes. (Bottom Left) Loss vs. cumulative work units,  (Bottom Right) Maximum GPU memory usage per training epoch. AD \emph{cannot} be employed due to high memory requirements of backpropagating through each application of $T_\theta$.}}
        % \texttt{CVXPYLayers} \emph{cannot} be used here due to non-convexity in the dynamics of bicycles and automatic differentiation (AD) cannot be employed due to high memory requirements of backpropagating through each application of $T_\theta$.}}
        \label{fig:20_bikes_JFB}
    \end{figure}  

    \noindent \textbf{Multi-Bicycle Dynamics (High-Dimensional).} The bicycle dynamics are given by $\dot{x} = v\cos(\psi)$, $\dot{y} = v\sin(\psi)$, $\dot{\psi} = \frac{v}{K}\tan(u_1)$, and $\dot{v} = u_2$, where $(x,y)$ represents position, $\psi$ is heading angle, $v$ is velocity, $K$ is wheelbase, and controls are steering angle $u_1$ and acceleration $u_2$. We employ a quadratic running cost $L(z,u) = \|z - z_{\text{ref}}(t)\|^2$ for trajectory tracking.
    Unlike the quadrotor case, the bicycle dynamics are not affine in the control due to the nonlinear steering relationship $\tan(u_1)$, which creates an implicit Hamiltonian. This nonlinearity prevents closed-form solutions and provides a complementary test case.
    Standard implicit differentiation solvers such as \texttt{CVXPYLayers}~\cite{agrawal2019differentiable} \emph{cannot} be used in this problem due to non-convexity in the bicycle dynamics. We evaluate our method on 5 and 20 bicycles. 
    For 5 bicycles, we have 20-dimensional state and 10-dimensional control. For 20 bicycles, we have 80-dimensional states and 40-dimensional controls.
    
    % Traditional AD could not be used on larger versions of this problem because the excessive memory consumption tiggered an out-of-memory runtime exception. JFB, however, was able to solve the problem with 100 bicycles (400 state dimensions and 200 control dimensions).

    \subsection{Training Setup}
    For each problem, we compare three gradient computation approaches: the true gradient using \texttt{CVXPYLayers}~\cite{agrawal2019differentiable} (when applicable), our proposed JFB~\cite{fung2022jfb}, and automatic differentiation (AD) through the entire fixed point iteration. We parameterize value functions using fully connected networks with 4 hidden layers of 128 units each and anti derivative of $\tanh$ as activations. The fixed point operator uses step size $\alpha = 0.1$ for the quadcopter and $\alpha = 5.0 \times 10^{-4}$ for 5 bicycles and $10^{-4}$ for 20 bicycles with convergence tolerance $10^{-3}$ for the quadcopter and $10^{-4}$ for the bicycles. 

    For the quadrotor, we train for 600 epochs using a constant learning rate of $10^{-3}$ with batch size 50. For 5 bicycles, we train for 150 epochs using the \texttt{ReduceOnPlateau} learning rate scheduler with initial learning rate of $10^{-2}$ and batch size 300. For the 20-bicycle problem, we train for 1500 epochs using the \texttt{ReduceOnPlateau} learning rate scheduler with initial learning rate of $5 \times 10^{-3}$ and batch size 200. 
    % We evaluate on held-out test trajectories and measure both control performance (final cost) and computational efficiency (training time, GPU memory usage). 
    Additionally, we compare the computed loss by each method against the number of work units needed to get that loss value, with one work unit being one evaluation of $\frac{\partial T_{\theta}}{\partial \theta}$, which is equivalent to backpropagation through one application of $T_\theta$. Since \texttt{CVXPYLayers} is a black-box package, it is not possible to compute work units for that method (and we compare with \texttt{CVXPYLayers} only using loss vs. epoch and loss vs. runtime). Each experiment was run 3 times with the solid curve representing the mean and the shaded region covering the minimum and maximum values of loss (or memory) over epoch (or time).

    \subsection{Results}
    For the quadrotor experiment in Fig.~\ref{fig:single_quadrotor}, all methods demonstrate a similar level of accuracy per epoch. However, \texttt{CVXPYLayers} requires significantly more time to converge as it requires solving the a linear system arising from~\eqref{eq:implicit_gradient}. We also show the memory consumption for each method. 
    % We note that while \texttt{CVXPYLayers} requires the most memory, it remains constant as the dimension of the problem grows. 
    For the 5-bicycle example in Fig.~\ref{fig:5_bikes}, we are unable to use \texttt{CVXPYLayers} due to non-convex functions in the dynamics. Our results show that JFB and AD perform similarly in terms loss vs. epoch. However, AD takes significantly more time and memory; this is because AD requires backpropagating through each application of $T_\theta$ during the fixed point iteration, which requires significantly more memory~\cite{fung2022jfb}. JFB is also able to achieve a given loss value in multiple orders of magnitude fewer work units than AD. Finally, for the 20 bicycle experiment in Fig.~\ref{fig:20_bikes_JFB}, AD cannot be used due to memory constraints. Worth noting, JFB uses roughly the same memory consumption for both, 5 and 20 bicycles; again, this is because it is based on implicit differentiation which is known to be constant in memory~\cite{fung2022jfb}. Trajectories for an instance of 20 bicycles is shown in Fig.~\ref{fig:20_bikes}. 
    In summary, we observe promising results for JFB as it is the fastest and least memory-consuming method for all experiments.

    \begin{figure}[t]
        \centering
        \textbf{20 Bicycle Trajectories}
        \\
        \includegraphics[width=0.8\linewidth]{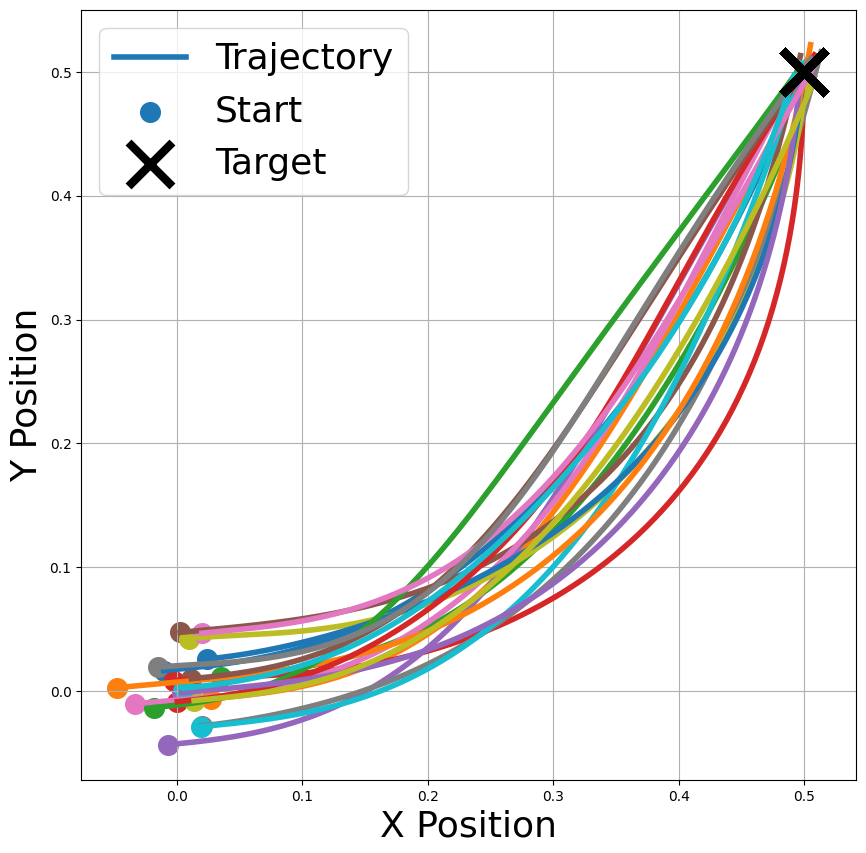}
        \caption{\small{Trajectories for an instance of the 20-bicycle problem. This high-dimensional problem causes memory issues with automatic differentiation (AD), while \texttt{CVXPYLayers} cannot be applied due to non-convex dynamics.}}
        \label{fig:20_bikes}
    \end{figure}
    
    \section{Conclusion}
    We introduce an implicit deep learning approach for learning high-dimensional feedback controllers when the Hamiltonian lacks closed-form solutions. 
    While existing methods based on implicit (or automatic) differentiation can handle such problems in principle, they become computationally prohibitive for high-dimensional multi-agent control scenarios. Our approach leverages Jacobian-Free Backpropagation (JFB) to enable fast and efficient training while preserving the structural advantages of value function parameterization.
    
    We extend the theoretical foundations of JFB to the optimal control setting, showing that it provides descent directions for the control objective despite the temporal coupling inherent in the trajectories. 
    Our experiments demonstrate that JFB matches or exceeds the performance of traditional approaches while offering significant computational advantages. 
    This combination of theoretical guarantees and computational efficiency makes JFB well-suited for learning feedback controllers via value function methods in high-dimensional control problems with implicit Hamiltonians. Future work will extend this framework to mean-field control/games settings~\cite{lasry2007mean, vidal2025kernel, agrawal2022random, lauriere2022learning, wang2025primal, chow2022numerical, yang2023relative}, where the scalability of our approach may prove useful.
\appendix

\section{Appendix}
\label{sec:appendix}
% \subsection{Proofs}
For readability, we restate the statements we prove here. 
Before proving the main theorem, we first show the following Lemmas.

% \paragraph{Proof of Lemma~\ref{lemma:T-contr}}
% \begin{proof}
% Consider
% \begin{align*}
%     \|T(u) - T(v)\|^2 &= \|(u + \alpha \nabla \mathcal{H}(u)) - (v + \alpha \nabla \mathcal{H}(v))\|^2 
%     = \|(u - v) + \alpha(\nabla \mathcal{H}(u) - \nabla \mathcal{H}(v))\|^2 \\
%     &= \|u - v\|^2 + 2\alpha \langle u - v, \nabla \mathcal{H}(u) - \nabla \mathcal{H}(v) \rangle + \alpha^2 \|\nabla \mathcal{H}(u) - \nabla \mathcal{H}(v)\|^2.
% \end{align*}
% Using strong concavity of $\mathcal{H}$ and Lipschitz continuity of $\nabla \mathcal{H}$ yields
% \begin{align*}
%     \|T(u) - T(v)\|^2 &\le \|u - v\|^2 - 2\alpha\mu \|u - v\|^2 + \alpha^2 L^2 \|u - v\|^2 \\
%     &= (1 - 2\alpha\mu + \alpha^2 L^2) \|u - v\|^2.
% \end{align*}
% For the operator $T$ to be a contraction, we require $(1 - 2\alpha\mu + \alpha^2 L^2)$ to be less than 1. That it, if $2\alpha\mu - \alpha^2 L^2 > 0 \implies  \alpha < \frac{2\mu}{L^2}$. 
% By choosing a step size $\alpha \in (0, 2\mu/L^2)$, the contraction factor $\gamma := \sqrt{1 - 2\alpha\mu + \alpha^2 L^2} \in [0, 1)$, proves the result.
% \end{proof}

\begin{lemma}
    Let $x \in \mathbb{R}^d$ such that $||x|| \geq c_1$ for some $c_1 > 0$. Let $A \in \mathbb{R}^{d \times d}$ be nonsingular. Then, $\exists\,  c_2 > 0$ such that $||Ax|| \geq c_2$ where $||\cdot||$ denotes the vector 2-norm.
\end{lemma}
\begin{proof}
    Because $A$ is nonsingular, by the Singular Value Decomposition $\exists \text{ orthogonal }U,V \in \mathbb{R}^{d \times d}$ and $\Sigma = \text{diag}(\sigma_1,...,\sigma_d)$ such that $A = U \Sigma V^T$ where the singular values $\sigma_j, 1 \leq j \leq d$ satisfy $\sigma_1 \geq ... \geq \sigma_d > 0$. Let $V_1,...,V_d$ denote the columns of $V$. Then, using the orthogonality of $U,V$
\begin{small}
\begin{align*}
    ||Ax|| =  ||U\Sigma V^Tx|| &= ||\Sigma V^Tx|| 
    = \sqrt{\sum_{j=1}^{d}(\sigma_j\langle V_j,x\rangle})^2  
    \\
    &\geq \sqrt{\sigma_{d}^{2}\sum_{j=1}^{d}(\langle V_j,x\rangle})^2  \geq \sigma_d ||x||.
\end{align*}
\end{small}
If $||x|| \geq c_1 > 0$ then it follows that $||Ax|| \geq c_2 > 0$ where $c_2 = \sigma_dc_1$. 
\end{proof}

\newcommand{\lemmaPositiveIntegralProduct}[1]{ 
Let $v,w: [0,T] \rightarrow \mathbb{R}^p$ satisfy Assumption ~\ref{assumption_timeaverage}. If $\langle v(t), w(t) \rangle \geq \delta^2$ for all $t,z,u,\theta$, then $\left(\int_{0}^{T}v(t)dt\right)^\top \left(\int_{0}^{T}w(t)dt\right) > 0$ for all $z,u,\theta$.}  
%\begin{enumerate}
%    \item $\exists B > 0$ such that $\langle v(t), w(t) \rangle \geq B$, 
%    \item B satisfies
%    \begin{equation}
%    B > \frac{1}{T}P_{v,max}P_{w,max}\left\| \frac{\partial v}{\partial t}\right\|_{\dagger}\left\| \frac{\partial w}{\partial t}\right\|_{\dagger}
%    \end{equation}
%\end{enumerate}
\begin{lemma}
\label{lemma:positive_integral_product}
\lemmaPositiveIntegralProduct{main}
\end{lemma}

%\textbf{Outline of Proof:}
%\begin{enumerate}
%    \item[Step 1.] Derive an expression for $\left(\int_{0}^{T}v(t)dt\right)^\top \left(\int_{0}^{T}w(t)dt\right)$ in terms of $\int_{0}^{T}v(t)^{\top}w(t)dt$ and $\int_{0}^{T}(v(t) - C_v)^{\top}(w(t) - C_w)dt$ to establish a sufficient condition such that the desired result is true.
%    \item[Step 2.] Use the Cauchy-Schwarz and Poincare-Wirtinger inequalities to find an upper bound for $\int_{0}^{T}(v(t) - C_v)^{\top}(w(t) - C_w)dt$
%    \item[Step 3.] Use the results of Steps 1 and 2 to establish the desired result for $\left(\int_{0}^{T}v(t)dt\right)^\top \left(\int_{0}^{T}w(t)dt\right)$ 
%\end{enumerate}
\begin{proof} 
Let $I_1 = \int_{0}^{T}v(t)dt$ and $I_2 = \int_{0}^{T}w(t)dt$, so $I_{1}^{\top}I_2 = T^2 C_{v}^{\top}C_{w}$. 

\begin{small}
\begin{align*}
\int_{0}^{T}v^{\top}(t)w(t)dt &= \int_{0}^{T}(v(t) - C_v + C_v)^{\top}(w(t) - C_w + C_w)dt \\
&= \int_{0}^{T}(v(t) - C_v)^{\top}(w(t) - C_w)dt + TC_{v}^{\top}C_{w}
\end{align*}
\end{small}
where we used $\int_0^T (v(t) - C_v) dt = 0 = \int_0^T (w(t) - C_w) dt$. Therefore,
\begin{small}
\begin{equation}
I_{1}^{\top}I_2 = T\left(\int_{0}^{T}v^{\top}(t)w(t)dt - \int_{0}^{T}(v(t) -C_v)^{\top}(w(t)-C_w)dt\right)
\label{eq:inner_product_lemma4}
\end{equation}
\end{small}
By Cauchy-Schwarz and Assumption~\ref{assumption_timeaverage}, 
$\langle v(t) - C_v, w(t) - C_w \rangle \leq ||v-C_v||||w-C_w|| < \delta^2$. 

Since $\langle v(t),w(t) \rangle \geq \delta^2$ by assumption, we have $\langle v(t),w(t) \rangle > \langle v(t) - C_v,w(t) - C_w \rangle$, thus $I_{1}^{\top}I_2 > 0$.
\end{proof}

\noindent \textbf{Proof of Lemma~\ref{lemma:positive_inner_product}.}
\noindent
\textit{\lemmaPositiveInnerProduct{app}}
\textit{Outline of Proof:} The proof is carried out in three main steps.
\begin{itemize}
    \item[] \textbf{Step 1.} Show that $\exists c > 0$ such that $M_{\theta} v_{\theta} \geq c$.
    \item[] \textbf{Step 2.} Reformulate $\langle v_{\theta}, w_{\theta}\rangle$ in terms of $M_{\theta} v_{\theta}$.
    \item[] \textbf{Step 3.} Use this new formulation and the given assumptions to derive the desired inequality. 
\end{itemize}

\begin{proof} 
For $v_\theta = M_{\theta}^\top \Jcal_{\theta}\invT h$ and $w_{\theta} = M_{\theta}\T h$, where $h = \nabla_u L + \nabla_u f\T p$ and $\Jcal_{\theta} = I - \frac{\partial T_{\theta}}{\partial u}$ as defined in \eqref{eq:implicit_gradient}. We are given that the matrix $M_{\theta}$ has full row rank, $\Jcal$ is nonsingular, and $h$ is uniformly bounded below in norm, i.e., $\norm{h} \geq \eta > 0$. We wish to show that $\inner{v_\theta}{w_\theta}$ is bounded below by a positive constant.

\textit{Step 1}: Let $\psi \coloneqq M_{\theta}v_{\theta}$. Substituting the definition of $v_{\theta}$:
\begin{small}
$$\psi = M_{\theta}(M_{\theta}^\top \Jcal_{\theta}\invT h) = (M_\theta M_\theta^\top)\Jcal_{\theta}\invT h.$$
\end{small}
Since $M_{\theta}$ has full row rank, the matrix $M_\theta M_\theta^\top$ is symmetric and positive definite. The matrix $\Jcal\invT$ is invertible and $\|\Jcal_{\theta}\| = \|I - \frac{\partial T_{\theta}}{\partial u}\| \leq 1+\gamma$ by contractivity of $T_\theta$, thus $\sigma_{\min} (\Jcal_{\theta}\invT) = \frac{1}{\sigma_{\max} (\Jcal_{\theta})} \geq \frac{1}{1+\gamma}$.

From Assumption~\ref{assumption_2}, we have $\sigma_{\min}(M_\theta M_\theta^\top) > \frac{\gamma}{\beta}$. Consequently:
\begin{small}
\begin{equation*}
\begin{aligned}
\norm{\psi} = \norm{(M_\theta M_\theta^\top)\Jcal_{\theta}\invT h} & \geq \sigmin(M_\theta M_\theta^\top)\sigmin(\Jcal_{\theta}\invT) \norm{h} \\
& > \frac{\gamma \eta}{\beta(1+\gamma)} >0.
\end{aligned}
\label{eq: psinorm}
\end{equation*}
\end{small}

\textit{Step 2}: Our goal is to express $\inner{v_\theta}{w_\theta}$ in terms of $\psi$. From Step 1, we have $\psi = (M_\theta M_\theta^\top)\Jcal_\theta\invT h$, so:
\begin{small}
$$h = \Jcal_{\theta}\T(M_\theta M_\theta^\top)^{-1}\psi.$$
\end{small}
Substituting this back into the definitions:
\begin{small}
\begin{align*}
v_{\theta} &= M_{\theta}^\top \Jcal_{\theta}\invT[\Jcal_{\theta}\T(M_\theta M_\theta^\top)^{-1}\psi] = M\T(M_\theta M_\theta^\top)^{-1}\psi \\
w_{\theta} &= M\T[\Jcal_{\theta}\T(M_\theta M_\theta^\top)^{-1}\psi] = M\T\Jcal_\theta\T(M_\theta M_\theta^\top)^{-1}\psi
\end{align*}
\end{small}
Therefore, $\inner{v_\theta}{w_\theta} = \inner{\psi}{\Jcal_\theta \T(M_\theta M_\theta^\top)^{-1}\psi}$, using the definition of adjoint.

\textit{Step 3}: Let $A = (M_\theta M_\theta^\top)^{-1}$. Since $M_\theta M_\theta^\top$ is symmetric positive definite, so is $A$. Let $\lambda_{+}$ and $\lambda_{-}$ denote the largest and smallest eigenvalues of $A$, and define $\bar{\lambda} = \frac{1}{2}(\lambda_{+} + \lambda_{-})$.

Using the assumption of this lemma and (Lemma A-1 in \cite{fung2022jfb}) ,  $\Jcal_\theta\T$ is corecive , that is,  $\inner{\psi}{\Jcal_\theta\T\psi} \geq (1-\gamma)\norm{\psi}^2$. Using this and Cauchy-Schwarz,
\begin{small}
\begin{align*}
\inner{\psi}{\Jcal_\theta\T A\psi} &= \inner{\psi}{\Jcal_\theta\T(\bar{\lambda}I + A - \bar{\lambda}I)\psi} \bar{\lambda}\inner{\psi}{\Jcal_\theta\T\psi} 
\\
&+ \inner{\psi}{\Jcal_\theta\T(A - \bar{\lambda}I)\psi}. \\
&\geq \bar{\lambda}(1-\gamma)\norm{\psi}^2 - \norm{\Jcal_\theta\T}\norm{A - \bar{\lambda}I}\norm{\psi}^2 
\end{align*}
\end{small}
Using \cite[Lemmas A-1 and A-2]{fung2022jfb}, we have $\norm{\Jcal_\theta\T} \leq 1+\gamma$ and $\norm{A - \bar{\lambda}I} = \frac{1}{2}(\lambda_{+} - \lambda_{-})$. Substituting these gives,
\begin{small}
\begin{align*}
\inner{v_\theta}{w_\theta} \geq \left[ \frac{\lambda_{+} + \lambda_{-}}{2}(1-\gamma) - (1+\gamma)\frac{\lambda_{+} - \lambda_{-}}{2} \right] \norm{\psi}^2 .
\end{align*}
\end{small}
Since the condition number $\kappa(A) < \frac{1}{\gamma}$ by assumption, we have $(\lambda_{-} - \gamma\lambda_{+}) > 0$. Combined with the lower bound from Step 1, we have
$\inner{v_\theta}{w_\theta} \geq \norm{\psi}^2(\lambda_{-} - \gamma\lambda_{+}) =: \delta^2 > 0.$
\end{proof}
\textbf{Proof of Theorem ~\ref{theorem:main_theorem} (Main Theorem)}
\noindent
\textit{\mainTheorem{app}}
\begin{proof}
% [Proof of Theorem ~\ref{theorem:main_theorem} (Main Theorem)]
The argument follows from the results of Lemmas ~\ref{lemma:positive_inner_product}-~\ref{lemma:positive_integral_product}.

By assumptions~\ref{assumption_1}-\ref{assumption_2}, the operator $T_{\theta}$ satisfies the necessary conditions, and the vector $h(t) = \nabla_uL + \nabla_uf\T p$ is not identically zero. Then, by Lemma~\ref{lemma:positive_inner_product},
$\inner{v_\theta(t)}{w_\theta(t)} \geq \delta^2$ for all $t,z,u,\theta$, 
where $\delta = ||M_{\theta}v_{\theta}||\sqrt{\lambda_{-} - \gamma\lambda_{+}}$ is shown to be bounded away from 0 in the proof of Lemma ~\ref{lemma:positive_inner_product}.
This result, combined with Assumptions ~\ref{assumption_timeaverage}, satisfies all of the hypotheses of Lemma~\ref{lemma:positive_integral_product}, ensuring that the inner product of the time-integrated vectors is also positive for all $z,u,\theta$:\\
$
\inner{\int_{0}^{T}v_\theta(t)dt}{\int_{0}^{T}w_\theta(t)dt} > 0.$
% Because the operator $T_{\theta}$ satisfies Assumptions 1-3 and $ \exists \eta > 0$ such that $||h|| = ||\nabla_uL + \nabla_uf^{\top}p|| \geq \eta$ for all $t,z,u,\theta$ is not identically equal to 0 by the assumptions of Lemma 3, it follows by Lemma 3 that $\exists B > 0$ such that $\langle v(t),w(t) \rangle \geq B$ for all $t,z,u,\theta$. Then, by Lemma 2 $\langle\int_{0}^{T}v(t)dt,\int_{0}^{T}w(t)dt \rangle > 0$ for all $z,u,\theta$ and hence JFB gives a descent direction. QEMFD \newline
% \newline \newline
\end{proof}

\section*{ACKNOWLEDGMENT}
SWF and EG were partially funded by NSF Award 2309810. SO was partially funded by DARPA under grant HR00112590074, NSF under grant 2208272 and 1554564, AFOSR under MURI grant N00014-20-1-278, and by ARO under grant W911NF-24-1-015. We also thank Howard Heaton and Levon Nurbekyan for meaningful discussions.

%%%%%%%%%%%%%%%%%%%%%%%%%%%%%%%%%%%%%%%%%%%%%%%%%%%%%%%%%%%%%%%%%%%%%%%%%%%%%%%%

% References are important to the reader; therefore, each citation must be complete and correct. If at all possible, references should be commonly available publications.

% \begin{thebibliography}{99}
\bibliographystyle{abbrv}
\bibliography{references}

\end{document}